\documentclass[11pt]{amsart}
\author{Christian J. Berghoff}
\title{Universal elliptic Gau{\ss} sums and applications}
\address{Universit\"at Bonn, Mathematisches Institut, Endenicher Allee 60, 53115 Bonn, Germany}
\email{berghoff@math.uni-bonn.de}

\usepackage{lmodern}
\usepackage[T1]{fontenc}
\usepackage[utf8x]{inputenc}
\usepackage{a4wide}
\usepackage[english]{babel}
\usepackage{amsmath, amssymb, amsthm, amsfonts}
\usepackage{framed}
\usepackage{enumerate}

\usepackage{algorithmic}

\algsetup{linenodelimiter=.}

\makeatletter
\newcounter{algorithm}
\renewcommand{\thealgorithm}{\arabic{algorithm}}
\def\algorithm{\@ifnextchar[{\@algorithma}{\@algorithmb}}
\def\@algorithma[#1]{%
	\refstepcounter{algorithm}
	\trivlist
	\leftmargin\z@
	\itemindent\z@
	\labelsep\z@
	\item[\parbox{\columnwidth}{%
		\hrule
		\hrule
		\noindent\strut\textbf{Algorithm \thealgorithm.} #1
		\hrule
	}]\hfil\vskip0em%
}
\def\@algorithmb{\@algorithma[]}

\makeatother

\usepackage{mathabx}
\usepackage[all]{xy}
\usepackage{textcomp}
\usepackage{stmaryrd}
\usepackage{nicefrac}
\usepackage[numbers]{natbib}
\usepackage[pdftex]{graphicx}
\usepackage{layout}
\usepackage{fancyhdr}
\usepackage{setspace}
\usepackage{color}
\usepackage{hyperref}

\usepackage{float}
\usepackage{pgfplots}
\pgfplotsset{compat=1.8}
\usepackage{subfigure}

\numberwithin{equation}{section}

\theoremstyle{plain}

\newtheorem{satz}{Theorem}[section]
\newtheorem{prop}[satz]{Proposition}
\newtheorem{lemma}[satz]{Lemma}
\newtheorem{koro}[satz]{Corollary}

\theoremstyle{definition}

\newtheorem{defi}[satz]{Definition}

\theoremstyle{remark}
\newtheorem{bem}[satz]{Remark}

\DeclareMathOperator{\gal}{Gal}

\DeclareMathOperator{\Tr}{Tr}
\DeclareMathOperator{\deg2}{deg}
\DeclareMathOperator{\ord}{ord}

\DeclareMathOperator{\SL}{SL}

\newcommand{\OO}{\mathcal{O}}
\newcommand{\Q}{\mathbb{Q}}
\newcommand{\F}{\mathbb{F}}

\newcommand{\N}{\mathbb{N}}
\newcommand{\Z}{\mathbb{Z}}

\newcommand{\C}{\mathbb{C}}

\begin{document}

\begin{abstract}
We present new ideas for computing the elliptic Gau{\ss} sums introduced in \cite{Mi1, Mi2} which constitute an analogue of the classical cyclotomic Gau{\ss} sums and whose use has been proposed in the context of counting points on elliptic curves and primality tests \cite{Mi_CIDE, CIDE}. By means of certain well-known modular functions we define the universal elliptic Gau{\ss} sums and prove they admit an efficiently computable representation in terms of the $j$-invariant and another modular function. After that, we show how this representation can be used for obtaining the elliptic Gau{\ss} sum associated to an elliptic curve over a finite field $\F_p$, which may then be employed for counting points or primality proving.
\end{abstract}
\maketitle

\tableofcontents
\section{Elliptic curves}\label{sec:ell_kurven}
Within this work we will only consider primes $p>3$ and thus assume that the curve in question is given in the Weierstra{\ss} form
\[ 
E: Y^2=X^3+aX+b=f(X),
 \]
where $a, b \in \F_p$. We will always identify $E$ with its set of points $E(\overline{\F_p})$. For the following well-known statements cf. \cite{Silverman, Washington}. We assume that the elliptic curve is neither singular nor supersingular. It is a standard fact that $E$ is an abelian group with respect to point addition. Its neutral element, the point at infinity, will be denoted $\OO$. For a prime $\ell \neq p$, the $\ell$-torsion subgroup $E[\ell]$ has the shape
\[ 
E[\ell]\cong \Z/\ell\Z \times \Z/\ell\Z.
 \]
In the endomorphism ring of $E$ the Frobenius homomorphism
\[ 
\phi_p: (X, Y) \mapsto (\varphi_p(X), \varphi_p(Y))=(X^p, Y^p)
 \]
satisfies the quadratic equation 
\begin{equation}\label{eq:char_gl}
0=\chi(\phi_p)=\phi_p^2-t\phi_p+p,
 \end{equation} 
where $|t|\leq 2\sqrt{p}$ by the Hasse bound. By restriction $\phi_p$ acts as a linear map on $E[\ell]$. The number of points on $E$ over $\F_p$ is given by $\#E(\F_p)=p+1-t$ and is thus immediate from the value of $t$.\\
The idea of Schoof's algorithm now consists in computing the value of $t$ modulo $\ell$ for sufficiently many small primes $\ell$ by considering $\chi(\phi_p)$ modulo $\ell$ and in afterwards combining the results by means of the Chinese Remainder Theorem. In the original version this requires computations in extensions of degree $O(\ell^2)$.\\
However, a lot of work has been put into elaborating improvements. Let $\Delta=t^2-4p$ denote the discriminant of equation \eqref{eq:char_gl}. Then we distinguish the following cases:
\begin{enumerate}
\item If $\left( \frac{\Delta}{\ell}\right)=1$, then $\ell$ is called an \textit{Elkies prime}. In this case, the characteristic equation factors as $\chi(\phi_p)=(\phi_p-\lambda)(\phi_p-\mu) \mod \ell$, so when acting on $E[\ell]$ the map $\phi_p$ has two eigenvalues $\lambda, \mu \in \F_\ell^*$ with corresponding eigenpoints $P, Q$. Since $\lambda\mu=p$ and $\lambda+\mu=t$, it suffices to determine one of them by solving the discrete logarithm problem
\[ 
\lambda P=\phi_p(P)=(P_x^p, P_y^p),
 \]
which only requires working in extensions of degree $O(\ell)$.
\item If  $\left( \frac{\Delta}{\ell}\right)=-1$, then $\ell$ is called an \textit{Atkin prime}. In this case the eigenvalues of $\phi_p$ are in $\F_{\ell^2}\backslash\F_\ell$ and there is no eigenpoint $P \in E[\ell]$. There is a generic method for computing the value of $t \mod \ell$ for Atkin primes, which is of equal run-time as the one available for Elkies primes. However, it does not yield the exact value of $t \mod \ell$ but only a set of candidates and is thus only efficient provided the cardinality of this set is small.
\end{enumerate}

The approach to Elkies primes was further improved in numerous publications, e.~g. \cite{MaMu, GaMo, BoMoSaSc, Enge, BrLaSu, Sutherland}. We focus on the new ideas introduced in \cite{MiMoSc}. The algorithm it presents allows to work in extensions of degree $n$, where $n$ runs through maximal coprime divisors of $\ell-1$, using so-called \textit{elliptic Gaussian periods}.\\

A variant of this approach was presented in \cite{Mi1} and \cite{MiVu}. It relies instead on so-called \textit{elliptic Gau{\ss} sums}. For a character $\chi: (\Z/\ell\Z)^* \rightarrow \langle\zeta_n\rangle$ of order $n$ with $n \mid \ell-1$ these are defined in analogy to the classical cyclotomic Gau{\ss} sums via
\begin{equation}\label{eq:ell_gs}
G_{\ell,n, \chi}(E)=\sum_{a=1}^{\ell-1}\chi(a)(aP)_V
\end{equation}
for an $\ell$-torsion point $P$ on $E$, where $V=y$ for $n$ even and $V=x$ for $n$ odd. As was shown in \cite{Mi1},
\begin{equation}\label{eq:ell_gs_pot_eigenschaft}
G_{\ell,n, \chi}(E)^n, \frac{G_{\ell,n, \chi}(E)^m}{G_{\ell,n, \chi^m}(E)} \in \F_p[\zeta_n]\quad \text{for}\quad m<n
\end{equation}
holds. In addition, the discrete logarithm in $\F_\ell^*$ of the eigenvalue $\lambda$ corresponding to $P$ can directly be calculated modulo $n$ using the equation
\begin{equation}\label{eq:lambda_aus_ell_gs}
G_{\ell,n, \chi}(E)^p=\chi^{-p}(\lambda)G_{\ell, n, \chi^p}(E)\quad \Rightarrow \quad 
\frac{G_{\ell, n, \chi}(E)^m}{G_{\ell, n, \chi^m}(E)}(G_{\ell, n, \chi}(E)^n)^q=\chi^{-m}(\lambda),
\end{equation}
where $p=nq+m$ holds. When the quantities from equation \eqref{eq:ell_gs_pot_eigenschaft} have been computed, it thus suffices to perform calculations in the extension $\F_p[\zeta_n]$ of degree $\varphi(n)$ to derive the discrete logarithm of $\lambda$ in $\F_\ell^*$ modulo $n$ before composing the modular information by means of the Chinese remainder theorem. In the following sections we will present a way to compute the quantities in question using \textit{universal elliptic Gau{\ss} sums}, which we will define in equation \eqref{eq:tau_ln_lr}, instead of using the definition \eqref{eq:ell_gs}, which requires passing through larger extensions.
\section{Universal elliptic Gau{\ss} sums}
\subsection{Modular functions}
In this section we recall some facts on modular functions which we will later use. We refer the reader to \cite{Apostol, Koblitz, Shimura}.\\

As usual, we denote $\mathbb{H}=\{\tau \in \C: \Im(\tau)>0 \}$ and $\Gamma:=\SL_2(\Z)$. Elements $\gamma=\left( \begin{smallmatrix}
a & b\\
c & d
\end{smallmatrix}\right)$
act on the upper complex half-plane via
\[ 
\gamma: \mathbb{H} \rightarrow \mathbb{H},\quad \tau \mapsto \frac{a\tau+b}{c\tau+d}.
 \]
For $N \in \N$
\begin{align*}
\Gamma_0(N)=&\left\{ \begin{pmatrix}
a & b \\
c & d
\end{pmatrix} \in \Gamma: c \equiv 0 \mod N \right\}\quad \text{and}\\
\Gamma(N)=&\left\{ \begin{pmatrix}
a & b \\
c & d
\end{pmatrix} \in \Gamma: a \equiv d \equiv 1 \mod N, b \equiv c \equiv 0 \mod N \right\}=\left\{
\begin{pmatrix}
1 & 0 \\
0 & 1
\end{pmatrix} \mod N \right\}
\end{align*}
are subgroups of $\Gamma$ that we will later use. 

\begin{defi}\cite[p.~125]{Koblitz}
Let $f(\tau)$ be a meromorphic function on $\mathbb{H}$, $k \in \Z$ and $\Gamma^{\prime}\leq \Gamma$, such that $\Gamma^{\prime} \supseteq \Gamma(N)$ for some $N \in \N$. Furthermore, let $f(\tau)$ satisfy the following conditions:
\begin{enumerate}
\item $f(\gamma\tau)=(c\tau+d)^kf(\tau)$ for all $\gamma=
\left(\begin{smallmatrix}
a & b\\ c & d
\end{smallmatrix}\right) \in \Gamma^{\prime}$. This implies in particular that
$f(\tau)$ may be written as a  Laurent series in terms of
\[
q_N=q^{\frac{1}{N}}=\exp\left(\frac{2\pi i \tau}{N}\right),\text{ where we use the notation } q=q_1.
 \]
\item In the Fourier expansion
\[ 
f(\gamma\tau)=\sum_{n \in \Z} a_nq_N^n
 \]
$a_n=0$ holds for $n<n_0$, $n_0 \in \Z$, for all $\gamma \in \Gamma$. One also says that $f(\tau)$ is meromorphic at the cusps.
\end{enumerate}
Then  $f(\tau)$ is called a \textit{modular function of weight} $k$ \textit{for} $\Gamma^{\prime}$. We denote by $\mathbf{A}_k(\Gamma^{\prime})$ the set of all such modular functions.
\end{defi}
\begin{bem}\label{bem:rep_system_reicht}
It suffices to check the second condition for a set of representatives of $\Gamma\backslash\Gamma^{\prime}$.
\end{bem}

\begin{defi}\cite[p.~112]{Iwaniec}\label{def:def_wl}
	Let $f(q)=f(\tau)$ be a modular function for $\Gamma^{\prime}\subseteq \Gamma$. Let $\ell$ be a prime. Then we define the Fricke-Atkin-Lehner involution $w_\ell$ by
	\[ 
	w_\ell:  f(\tau) \mapsto
	f \left( \left( \begin{matrix}
	0 & -1\\
	\ell & 0
	\end{matrix}\right) \tau \right)= f\left( -\frac{1}{\ell\tau} \right)=:f^*(\tau).
	\]
\end{defi}

\begin{bem}
	For $f(\tau) \in \mathbf{A}_0(\Gamma)$ this yields
	\[ 
	f^*(\tau)=f\left( -\frac{1}{\ell\tau} \right)=f \left(  \begin{pmatrix}
	0 & -1\\
	1 & 0
	\end{pmatrix} \ell\tau \right)=f(\ell\tau).
	\]
\end{bem}

We will make use of the following modular functions:
\begin{align}
E_4(\tau)&=E_4(q)=1+240\sum_{n=1}^{\infty} \frac{n^3q^n}{1-q^n},\label{eq:e4_lr}\\
E_6(\tau)&=E_6(q)=1-504\sum_{n=1}^{\infty} \frac{n^5q^n}{1-q^n},\label{eq:e6_lr}\\
\Delta(\tau)&=\frac{E_4(\tau)^3-E_6(\tau)^2}{1728},\label{eq:delta_lr}\\
j(\tau)&=\frac{E_4(\tau)^3}{\Delta(\tau)},\label{eq:j_lr}\\
p_1(q)&=\frac{1}{12}\ell(E_2(q)-\ell E_2(q^\ell)),\label{eq:p1_def_1}\\
m_\ell(q)&=\ell^s\left(\frac{\eta(q^\ell)}{\eta(q)}\right)^{2s}\quad \text{with}\quad s=\frac{12}{\gcd(12, \ell-1)}\label{eq:ml_lr}.
\end{align}
Here,
\begin{align}
E_2(q)&=1-24\sum_{n=1}^{\infty}\frac{nq^n}{1-q^n}\quad \text{and}\label{eq:e2_lr}\\
\eta(\tau)&=\eta(q)=q^{\frac{1}{24}}\prod_{n=1}^{\infty}(1-q^n)\label{eq:eta_lr}
 \end{align}
is the Dedekind $\eta$-function. The Eisenstein series $E_4$ and $E_6$ are modular functions of weight $4$ and $6$, respectively, for $\Gamma$. $\Delta$ is the discriminant of the elliptic curve $E_\tau$ corresponding to the lattice $\langle 1, \tau \rangle_\Z$ (cf. theorem \ref{satz:iso_kurve_gitter}) and $j$ is its $j$-invariant. They are likewise modular functions for $\Gamma$ of weight $12$ and $0$, respectively. $p_1$ is a modular function of weight $2$ for $\Gamma_0(\ell)$. The function $m_\ell$ was studied in detail in \cite{Mueller}, where it is shown to be a modular function of weight $0$ for $\Gamma_0(\ell)$. We remark that this already follows from general results on so-called eta-quotients established in \cite{Rademacher, Newman}.\\

The $j$-invariant $j: \mathbb{H} \rightarrow \C$ is surjective and plays a fundamental role in the theory of modular functions, as is shown by the following
\begin{satz}\label{satz:j_erzeugt_mod_funk}
Denoting by $\mathbf{H}_0(\Gamma)$ the subset of holomorphic functions of weight $0$ for $\Gamma$ we have
\[ 
\mathbf{A}_0(\Gamma)=\C(j(\tau))\quad\text{and}\quad \mathbf{H}_0(\Gamma)=\C[j(\tau)],
 \]
so the modular functions of weight $0$ are the rational functions in $j$, whereas the holomorphic ones are the polynomials in $j$.
\end{satz}

\subsection{Modular functions for \texorpdfstring{$\Gamma_0(\ell)$}{Gamma-0(l)}}
Our goal now is to prove the following statement, which can be seen as a generalization of theorem \ref{satz:j_erzeugt_mod_funk} to the group $\Gamma_0(\ell)$ and which is crucial for later considerations.\\

\begin{satz}\label{satz:j_f_erzeugen_mod_funk}
Let $f(\tau) \in \mathbf{A}_0(\Gamma_0(\ell))\backslash\mathbf{A}_0(\Gamma)$ be a modular function of weight $0$ for $\Gamma_0(\ell)$, but not for $\Gamma$. Then
\[ 
\mathbf{A}_0(\Gamma_0(\ell))=\mathbf{A}_0(\Gamma)(f(\tau))=\C(f(\tau), j(\tau))  
 \]
holds. In particular, for $g(\tau) \in \mathbf{A}_0(\Gamma_0(\ell))$ there exist polynomials $P_1, P_2 \in \C[X, Y]$ with
\[ 
g(\tau)=\frac{P_1(f(\tau), j(\tau))}{P_2(f(\tau), j(\tau))}.
 \]
\end{satz}

We shall proceed in several steps. First, we show the following
\begin{satz}\label{satz:galois_gruppe_allgemeines}
Let $\Gamma^{\prime} \unlhd \Gamma$ be a normal divisor of finite index in $\Gamma$. 
Then
$\mathbf{A}_0(\Gamma^{\prime})/\mathbf{A}_0(\Gamma)$ is a galois extension with
\[ 
\gal(\mathbf{A}_0(\Gamma^{\prime})/\mathbf{A}_0(\Gamma))\leq \Gamma/\Gamma^{\prime}.
 \]
\end{satz}
\begin{proof}
Let $f(\tau) \in \mathbf{A}_0(\Gamma^{\prime})$, then $f(\gamma\tau) \in \mathbf{A}_0(\Gamma^{\prime})$ holds for all $\gamma \in \Gamma$: Since $\Gamma^{\prime}$ is a normal divisor in $\Gamma$, the equation $\gamma\delta\gamma^{-1}=\tilde{\delta} \in \Gamma^{\prime}$ holds for $\delta \in \Gamma^{\prime}$. We deduce
\[ 
f(\gamma\delta\tau)=f(\tilde{\delta}\gamma\tau)=f(\gamma\tau)\quad \text{for all}\quad \delta \in \Gamma^{\prime}.
\]
Furthermore, $f(\gamma\tau)$ is meromorphic at the cusps, since $f(\tau)$ is. Replacing $\gamma$ by $\gamma^{-1}$, we see that
\[ 
\gamma^{*}: \mathbf{A}_0(\Gamma^{\prime}) \rightarrow \mathbf{A}_0(\Gamma^{\prime}),\quad f \mapsto f\circ \gamma
 \]
defines a bijection. Due to the invariance of the elements of $\mathbf{A}_0(\Gamma^{\prime})$ under $\pm\Gamma^{\prime}$ (since $-I$ induces the identity) and of $\mathbf{A}_0(\Gamma)$ under $\Gamma$ it follows that the finite group $\Gamma/(\pm\Gamma^{\prime})$ is the automorphism group of $\mathbf{A}_0(\Gamma^{\prime})$ and fixes $\mathbf{A}_0(\Gamma)$. Galois theory now implies our claim.
\end{proof}

A special case is given by
\begin{lemma}\cite[p.~134]{Shimura}
Let $N \in \N$. $\mathbf{A}_0(\Gamma(N))/\mathbf{A}_0(\Gamma)$ is a galois extension with
\[ 
\gal(\mathbf{A}_0(\Gamma(N))/\mathbf{A}_0(\Gamma))\cong \Gamma/(\pm\Gamma(N)).
 \]
\end{lemma}
We remark that $\Gamma(N)=\ker\{\SL_2(\Z) \xrightarrow{\mod N} \SL_2(\Z/N\Z)\}$ obviously implies the isomorphism
$\Gamma/\Gamma(N)\cong \SL_2(\Z/N\Z)$ and hence
\[ 
\gal(\mathbf{A}_0(\Gamma(N))/\mathbf{A}_0(\Gamma))\cong \SL_2(\Z/N\Z)/\{\pm 1\}.
 \]

\begin{koro}
Let $\Gamma^{\prime}$ be a subgroup of $\Gamma$ with $\Gamma(N)\leq \Gamma^{\prime}$, then
\[ 
\gal(\mathbf{A}_0(\Gamma(N))/\mathbf{A}_0(\Gamma^{\prime}))\cong (\pm\Gamma^{\prime})/(\pm\Gamma(N))
 \]
holds. In particular $\mathbf{A}_0(\Gamma^{\prime})$ is a finite extension of $\mathbf{A}_0(\Gamma)=\C(j)$ of degree $[\Gamma: \pm \Gamma^{\prime}]$.
\end{koro}
\begin{proof}
Since $\Gamma(N)$ is the kernel of the reduction map modulo $N$, $\Gamma(N)\unlhd\Gamma^{\prime}$ holds. This yields
$(\pm\Gamma^{\prime})/(\pm\Gamma(N))\leq \Gamma/(\pm\Gamma(N))
=\gal(\mathbf{A}_0(\Gamma(N))/\mathbf{A}_0(\Gamma))$. Since we have in addition
\[ 
\mathbf{A}_0(\Gamma(N))^{\Gamma^{\prime}}=\{f \in \mathbf{A}_0(\Gamma(N)): f \circ \gamma=f\ \forall \gamma \in \Gamma^{\prime}\}=\mathbf{A}_0(\Gamma{'}),
 \]
the statement concerning the galois group is implied by galois theory. Hence, we obtain $[\mathbf{A}_0(\Gamma(N)):\mathbf{A}_0(\Gamma{'})]=[\pm \Gamma^{\prime}: \pm \Gamma(N)]$ and $[\mathbf{A}_0(\Gamma{'}): \mathbf{A}_0(\Gamma)]=[\Gamma: \pm \Gamma^{\prime}]$.
\end{proof}
For $\Gamma^{\prime}=\Gamma_0(N)$ we glean
\[ 
\gal(\mathbf{A}_0(\Gamma(N))/\mathbf{A}_0(\Gamma_0(N)))\cong (\pm\Gamma_0(N))/(\pm\Gamma(N))\cong 
\left\{\begin{pmatrix}
a & b\\ 0 & d
\end{pmatrix} \in \SL_2(\Z/N\Z) \right\}/\{\pm  1\}.
 \]

Before proving theorem \ref{satz:j_f_erzeugen_mod_funk} we still need the following
\begin{lemma}\label{lem:keine_zwischengruppen}
Let $K$ be a field,
$B= \left\{   
\begin{pmatrix}
a & b \\
0 & d
\end{pmatrix} \mid a, d \in K^*, b \in K
\right\} \subseteq G=\SL_2(K)$.
Then
\begin{align*}
G/B &\rightarrow \mathbb{P}^1(K),\\
\ g \cdot B &\mapsto g \cdot \infty =g\cdot [1, 0], \text{ where } g\cdot [v] \mapsto [gv],\\
\text{i. e., } \begin{pmatrix}
a & b \\
c & d
\end{pmatrix}\cdot B &\mapsto [a, c]
\end{align*} 

defines a bijection and there are no intermediate groups between $G$ and $B$.

\begin{proof}
The first statement is trivial since the $G$-action on $\mathbb{P}^1(K)$ is transitive and $B$ is the stabiliser of $\infty$.\\
Since
\[ 
\begin{pmatrix}
1 & b \\
0 & 1
\end{pmatrix}[0, 1] \mapsto [b, 1],
 \]
the action of $B$ on $\mathbb{P}^1(K)\setminus \{\infty \}$ is transitive.\\

Now let $g, h \in G \setminus B$. Then $h \cdot \infty$ and $g \cdot \infty$ are contained in $\mathbb{P}^1(K)\setminus \{\infty\}$, as $B$  is the stabiliser of $\infty$. On this account there exists $b \in B$ such that
\[ 
h \cdot \infty= b \cdot g \cdot \infty,
 \]
hence $h^{-1}bg \in B$ or equivalently $h \in BgB$ holds. Thus, $B$ and $g$ generate $G$, which means there are no intermediate groups between $B$ and $G$.
\end{proof}
\end{lemma}

\begin{proof}[Proof of theorem \ref{satz:j_f_erzeugen_mod_funk}]
Let $\ell$ be a prime. The considerations above show
\[ 
G=\gal(\mathbf{A}_0(\Gamma(\ell))/\mathbf{A}_0(\Gamma))\cong \SL_2(\Z/\ell\Z)/\{\pm 1\}
 \]
as well as
\[ 
B=\gal(\mathbf{A}_0(\Gamma(\ell))/\mathbf{A}_0(\Gamma_0(\ell)))\cong\left\{\begin{pmatrix}
a & b\\ 0 & d
\end{pmatrix} \in \SL_2(\Z/\ell\Z) \right\}/\{\pm  1\}.
 \]
Hence, by galois theory the intermediate fields of the extension $\mathbf{A}_0(\Gamma_0(\ell))/\mathbf{A}_0(\Gamma)$ correspond exactly to the intermediate groups between $G$ and $B$. Applying lemma \ref{lem:keine_zwischengruppen} with $K=\F_\ell$ and observing that $-1 \in B$ holds, we deduce there are no intermediate groups between $G$ and $B$ and thus no intermediate fields between $\mathbf{A}_0(\Gamma_0(\ell))$ and $\mathbf{A}_0(\Gamma)$. This directly implies $\mathbf{A}_0(\Gamma_0(\ell))=\C(f(\tau), j(\tau))$ for any $f(\tau) \in \mathbf{A}_0(\Gamma_0(\ell))\backslash\mathbf{A}_0(\Gamma)$.
\end{proof}

Hence, modular functions of weight $0$ for $\Gamma_0(\ell)$, in particular the universal elliptic Gau{\ss} sums to be defined in corollary \ref{koro:gs_modulfunktion_2}, admit a representation as a rational expression in terms of $j(\tau)$ and another modular function $f(\tau) \in \mathbf{A}_0(\Gamma_0(\ell))\backslash \mathbf{A}_0(\Gamma)$. However, theorem \ref{satz:j_f_erzeugen_mod_funk} only implies the existence of such an expression. In order to obtain an efficient algorithm for determining it we will need further results. In addition we have to discuss the choice of the second function $f(\tau)$. For the following results we closely follow \cite[pp.~228--231]{Cox}. However, our results are slightly more general.

\begin{lemma}\label{lem:mod_funk_mi_po}
Let $f(\tau) \in \mathbf{A}_0(\Gamma_0(\ell))\backslash \mathbf{A}_0(\Gamma)$ be holomorphic on $\mathbb{H}$. Then there exists an irreducible polynomial $Q_f(X, Y) \in \C[X, Y]$ such that
\[ 
Q_f(f(\tau), j(\tau))=0.
 \]
\end{lemma} 
\begin{proof}
First, we remark that $\{S_k, k=0, \ldots, \ell\}$, where
\[ 
S_k=\begin{pmatrix}
0 & -1 \\ 1 & k
\end{pmatrix}\quad \text{for}\quad 0 \leq k < \ell,\quad
S_\ell=\begin{pmatrix}
1 & 0\\
0 & 1
\end{pmatrix},
 \]
is a system of representatives for $\Gamma/\Gamma_0(\ell)$, as is shown in \cite[p.~54]{Mueller}. We now consider the polynomial in $X$
\[ 
Q_f(X, \tau)=\prod_{k=0}^\ell (X-f(S_k\tau))
 \]
and examine its coefficients. Since they are elementary symmetric polynomials in terms of $f(S_k\tau)$ they are obviously holomorphic on $\mathbb{H}$. Let $\gamma \in \Gamma$. Since the $S_k$ constitute a system of representatives of $\Gamma/\Gamma_0(\ell)$, the values $f(S_k\gamma\tau)$, $k=0, \ldots, \ell$, are a permutation of the values $f(S_k\tau)$. Hence, the coefficients of $Q_f(X, \tau)$ are invariant under $\SL_2(\Z)$. The modular function $f(\tau)$ is meromorphic at the cusps, so this is also the case for $f(S_k\tau)$. Hence, the coefficients are meromorphic at the cusps and thus functions in $\mathbf{H}_0(\Gamma)$. According to theorem \ref{satz:j_erzeugt_mod_funk} they are therefore polynomials in $j(\tau)$. Thus, there exists a polynomial $Q_f(X, Y) \in \C[X, Y]$ satisfying
\[ 
Q_f(X, j(\tau))=\prod_{k=0}^\ell (X-f(S_k\tau)),
 \]
which obviously has $f(\tau)$ as one of its roots. Since there are no intermediate fields between $\mathbf{A}_0(\Gamma_0(\ell))$ and $\mathbf{A}_0(\Gamma)$ as we have seen, the polynomial $Q_f(X, j(\tau))$ has to be irreducible.
\end{proof}

The following statement, which may be proven by a generalisation of the considerations from \cite[pp.~230--231]{Cox} finally provides a first approach for an efficient algorithm.

\begin{satz}\label{satz:mod_funk_darstellung}
Let $g(\tau) \in \mathbf{A}_0(\Gamma_0(\ell))$ be a modular function and $f(\tau) \in \mathbf{H}_0(\Gamma_0(\ell))\backslash\mathbf{H}_0(\Gamma)$. Then $g$ admits the representation
\[ 
g(\tau)=\frac{Q(f(\tau), j(\tau))}{\frac{\partial Q_f}{\partial X}(f(\tau), j(\tau))},
 \]
where $Q(X, j(\tau)) \in \C(j(\tau))[X]$ is a polynomial in $X$ which can be explicitly specified in terms of
\[ 
\left\{g(S_k\tau), f(S_k\tau), i=0, \ldots, \ell\right\}.
 \]
If $g(\tau)$ is holomorphic, one even obtains $Q(X, j(\tau)) \in \C[j(\tau)][X]$; hence, the enumerator of the rational expression is a polynomial in $f$ and $j$.
\end{satz}
\begin{bem}
The case we are interested in, viz. when the function $g(\tau)$ is holomorphic, is also a direct consequence of lemma \ref{lem:neukirch-lemma} from \cite[pp.~206--208]{Neukirch}, which we present below.\\
\end{bem}

The representation from theorem \ref{satz:mod_funk_darstellung} is far from optimal from an algorithmic point of view since it does not allow to obtain good bounds on the powers of $j$, the coefficients of which grow very fast, occurring in the enumerator. We will rather make use of the following statements.

\begin{lemma}\cite[pp.~206--208]{Neukirch}\label{lem:neukirch-lemma}
Let $L/K$ be an extension of fields, $\OO \subseteq K$ be a ring. Let $\alpha \in L\backslash K$ have the minimal polynomial $f(X) \in K[X]$ of degree $n$. Then the $\OO$-module 
\[ 
C_{\alpha}=\{x \in L\mid \Tr_{L/K}(x\OO[\alpha]) \subseteq \OO\}
 \] 
has the $\OO$-basis 
\[ 
\left\{\frac{\alpha^i}{f^{\prime}(\alpha)}, i=0, \ldots, n-1\right\}.
 \]
\end{lemma}

\begin{koro}\label{koro:darst_besser_allgemein}
Let $g(\tau) \in \mathbf{H}_0(\Gamma_0(\ell))\backslash\mathbf{H}_0(\Gamma)$ be a holomorphic modular function and let $f(\tau)  \in \mathbf{H}_0(\Gamma_0(\ell))\backslash\mathbf{H}_0(\Gamma)$ with minimal polynomial $Q_f(X, j)$ with $\deg2_j(Q_f)=v$. Then $g(\tau)$ admits a representation of the form
\[ 
g(\tau)=\frac{\sum_{i=0}^{v-1}a_i j(\tau)^i}{\frac{\partial Q_f}{\partial Y}(f(\tau), j(\tau))},
 \]
where 
\[ 
a_i \in \{h(\tau) \in \C(f(\tau)): h(\tau)\text{ holomorphic}\}.
 \]
\end{koro}
\begin{proof}
We apply lemma \ref{lem:neukirch-lemma} for $K=\C(f(\tau))$, $L=K(j(\tau))=\mathbf{A}_0(\Gamma_0(\ell))$ and $\alpha=j(\tau)$ and hence $f(X)=Q_f(f(\tau), X)$. Furthermore, we choose
\[ 
\OO=\{h(\tau) \in K: h(\tau)\text{ holomorphic}\}.
 \]

Obviously, all elements $z$ of $\OO[j]$ and hence $g(\tau)z$ as well as $\Tr_{L/K}(g(\tau)z)$ are holomorphic. So $g(\tau) \in C_j$. Using lemma \ref{lem:neukirch-lemma} we obtain the assertion.
\end{proof}

We now consider special values for $f(\tau)$.
\begin{enumerate}
\item The most obvious and historically first choice is $f(\tau)=j(\ell\tau)$. In this case we write $Q_f(X, Y)=\Phi_\ell(X, Y)$ and call $\Phi_\ell$ the $\ell$-th modular polynomial. The modular polynomial has coefficients in $\Z$ and is symmetric in $X$ and $Y$ \cite[pp.~229--234]{Cox}. The main problem when using it is that its coefficients grow exponentially in $\ell$ \cite{Cohen}.
\item In \cite{Mueller} the choice $f(\tau)=m_\ell(\tau)$, where $m_\ell(\tau)$ is defined as in \eqref{eq:ml_lr}, was extensively examined and made applicable in the context of Schoof's algorithm. The construction of $m_\ell(\tau)$ and the relation between $\eta(\tau)$ and $\Delta(\tau)$ imply $m_\ell(\tau)$ is also holomorphic on $\mathbb{H}$. In this case we write $Q_f(X, Y)=M_\ell(X, Y)$. The polynomial $M_\ell$  has likewise coefficients in $\Z$ and the degree in the second variable is $v=\frac{\ell-1}{\gcd(12,\, \ell-1)}$, as is shown in \cite[pp.~61--62]{Mueller}. This causes its coefficients to grow much slower than those of $\Phi_\ell$.
\end{enumerate}

Using this specialisation we obtain the following
\begin{prop}\label{prop:darstellung_besser}
Let $g(\tau) \in \mathbf{H}_0(\Gamma_0(\ell))\backslash\mathbf{H}_0(\Gamma)$ be a holomorphic modular function. Then $g(\tau)$ admits the following representation:
\[ 
g(\tau)=\frac{Q(m_\ell(\tau), j(\tau))}{m_\ell(\tau)^k\frac{\partial M_\ell}{\partial Y}(m_\ell(\tau), j(\tau))}
 \]
for a $k\geq 0$ and a polynomial $Q(X, Y) \in \C[X, Y]$ with $\deg2_Y(Q)<v=\deg2_Y(M_\ell)$.

\end{prop}
\begin{proof}
We apply corollary \ref{koro:darst_besser_allgemein}, setting $f(\tau)=m_\ell(\tau)$. As mentioned, $m_\ell(\tau)$ is holomorphic, and this also holds for $m_\ell(\tau)^{-1}$. We now show $m_\ell(\mathbb{H})=\C^*$, which implies the holomorphic functions in $\C(m_\ell(\tau))$ are given by $\OO=\C[m_\ell(\tau), m_\ell(\tau)^{-1}]$. The assertion then follows using corollary \ref{koro:darst_besser_allgemein}.\\

By definition, $M_\ell(m_\ell(\tau), j(\tau))=0$ holds. Applying the Fricke-Atkin-Lehner involution $w_\ell$ from definition \ref{def:def_wl} to this equation we obtain that $w_\ell(m_\ell(\tau))$ is a root of $M_\ell(X, j(\ell\tau))$. Writing $M_\ell(X, Y)=\sum_{i=0}^{\ell+1}\sum_{k=0}^v a_{i, k}X^iY^k$ yields
\begin{equation}\label{eq:m_l_mipo_ell_sym}
\sum_{i=0}^{\ell+1}X^i\sum_{k=0}^v a_{i, k}j(\ell\tau)^k=M_\ell(X, j(\ell\tau))=\sum_{i=0}^{\ell+1}s_{\ell+1-i}(\tau)X^i,
 \end{equation}
where $s_{\ell+1-i}(\tau)$ are the elementary-symmetric polynomials in the roots $w_\ell(m_\ell(\tau))=f_0(\tau), \ldots,$ $f_\ell(\tau)$ of $M_\ell(X, j(\ell\tau))$. In \cite[p.~63]{Mueller} it is shown the Laurent series of these functions have the orders
\[ 
\ord(f_i)=-v,\ 0\leq i<\ell, \quad \ord(f_\ell)=\ell v,
 \]
from which we conclude
\[ 
\ord(s_{\ell+1-i})=-(\ell+1-i)v,\ 1\leq i\leq\ell+1, \quad \ord(s_{\ell+1})=0.
 \]
Using $\ord(j(\ell\tau))=-\ell$ equation \eqref{eq:m_l_mipo_ell_sym} now implies $a_{1, v}\neq 0$ and $a_{i, v}=0$ for $i\neq 1$. So $M_\ell(c, Y)$ is a polynomial of degree $v$ in $Y$ if $c \in \C^*$. Due to the surjectivity of $j$ there exists $\tau \in \mathbb{H}$ such that $M_\ell(c, j(\tau))=0$. Thus, for one of the transformations $S_k$ introduced in lemma \ref{lem:mod_funk_mi_po} the identity $c=m_\ell(S_k\tau)$ holds. Hence, $m_\ell$ attains all values $c \in \C^*$.
\end{proof}


\subsection{The Tate curve}\label{sec:tate_kurve}
In this section we will define the universal elliptic Gau{\ss} sums before considering applications in section \ref{sec:abschnitt_elkies}. We first recall the Weierstra{\ss} $\wp$-function.
\begin{defi}\cite[p.~200]{Cox}\label{def:wp_def}
Let $\Lambda=\langle \omega_1, \omega_2\rangle_\Z \subset \C$ be the lattice generated by $\omega_1, \omega_2$. The \textit{Weierstra{\ss} $\wp$-function} associated to $\Lambda$ is defined via
\[ 
\wp(z, \Lambda)=\frac{1}{z^2}+\sum_{\substack{\omega \in \Lambda\\ \omega\neq 0}} \left( \frac{1}{(z-\omega)^2}-\frac{1}{\omega^2} \right).
 \]
For the special lattices $\langle 1, \tau \rangle_\Z$ with $\tau \in \mathbb{H}$ we write
\[ 
\wp(z, \tau)=\wp(z, \langle 1, \tau \rangle_\Z)=\frac{1}{z^2}+\sum_{n^2+m^2\neq 0} \left( \frac{1}{(z-(m+n\tau))^2}-\frac{1}{(m+n\tau)^2} \right).
\]
Its derivative is
\[ 
\wp^{\prime}(z, \tau)=-\frac{2}{z^3}-2\sum_{n^2+m^2\neq 0} \frac{1}{(z-(m+n\tau))^3}.
\]
\end{defi}

\begin{satz}\cite[pp.~159--161]{Silverman}\label{satz:iso_kurve_gitter}
Let $E/\C$ be an elliptic curve. Then there exist $\tau \in \mathbb{H}$,  $\alpha \in \C$ such that setting $\Lambda=\langle 1, \tau \rangle_\Z$  there is a complex-analytic isomorphism
\[ 
\psi_1: \C/\Lambda \rightarrow E(\C), \quad z \mapsto \begin{cases}
(\wp(\alpha z, \alpha\Lambda), \wp^{\prime}(\alpha z,\alpha\Lambda)),\quad & z \notin \Lambda,\\
\OO,\quad & z \in \Lambda.
\end{cases}
 \]
\end{satz}

The following series expansions of the Weierstra{\ss} $\wp$-function will be used several times:
\begin{lemma}\cite[p.~50]{Silverman_advanced}\label{lem:x_y_def}
Let $q=e^{2\pi i \tau}$, $w=e^{2\pi i z}$. Then for $|q|<|w|<|q^{-1}|$ the following equations hold:
\begin{align}
\frac{1}{(2\pi i)^2}\wp(z, \tau)&=\frac{1}{12}-2\sum_{n=1}^{\infty}\frac{q^n}{(1-q^n)^2} +\sum_{n \in \Z} \frac{q^n w}{(1-q^n w)^2}=:x(w, q), \label{eq:x_def_alt}\\
\frac{1}{(2\pi i)^3 }\wp^{\prime}(z, \tau)&=\sum_{n \in \Z}\frac{q^n w(1+q^n w)}{(1-q^n w)^3}=:2y(w, q).
\end{align}
\end{lemma}

\begin{bem}
Using the equation for the geometric series and its derivatives the series expansions for $x(w,q)$ and $y(w, q)$ may be transformed, which is useful both for proofs and for actual computations. In this way for $|q|<|w|<|q^{-1}|$ we obtain the formulae
\begin{align}
x(w,q)&=\frac{1}{12}+\frac{w}{(1-w)^2}+\sum_{n=1}^{\infty}\sum_{m=1}^{\infty}mq^{nm}(w^m+w^{-m})
-2mq^{nm}, \label{eq:x_def}\\
y(w, q)&=\frac{w+w^2}{2(1-w)^3}+\frac{1}{2}\sum_{n=1}^{\infty}\sum_{m=1}^{\infty}\frac{m(m+1)}{2}
\left(q^{nm}(w^m-w^{-m})+q^{n(m+1)}(w^{m+1}-w^{-(m+1)}) \right). \label{eq:y_def}
\end{align}
In the cases we consider $w$ is a root of unity. Hence all series expansions can be studied in this form since $|q(\tau)|<1$ for $\tau \in \mathbb{H}$.
\end{bem}

\begin{prop}\cite[p.~245]{Schoof}
Let $E_4(q), E_6(q)$ be the Eisenstein series defined in \eqref{eq:e4_lr} and \eqref{eq:e6_lr} and $E_2(q)$ as in \eqref{eq:e2_lr}. Then the following equations hold:
\begin{align}
y(w, q)^2&=x(w, q)^3-\frac{E_4(q)}{48}x(w, q)+\frac{E_6(q)}{864}\label{eq:tate_kurve_def},  \\ 
\sum_{\zeta \in \mu_\ell, \zeta\neq 1}x(\zeta, q)&=\frac{1}{12}\ell(E_2(q)-\ell E_2(q^\ell))=p_1(q)\label{eq:p1_def}.
\end{align}
\end{prop}

Equation \eqref{eq:tate_kurve_def} defines the so-called Tate curve $E_q$ introduced in \cite{Tate}. Among its properties are the following ones:

\begin{satz}\cite[p.~410--411]{Silverman_advanced}\label{satz:tate_kurve_eig}
\begin{enumerate}
\item The Tate curve $E_q$ is an elliptic curve and
\[ 
\Delta(E_q)=\Delta(q),\quad j(E_q)=j(q)
 \]
holds, where $\Delta(q)$, $j(q)$ are the corresponding modular functions from equations \eqref{eq:delta_lr} and \eqref{eq:j_lr}.
\item There is a complex-analytic isomorphism
\[ 
\psi_2: \C^{\ast}/q^{\Z} \rightarrow E_q(\C),\quad 
w \mapsto \begin{cases}
(x(w, q), y(w, q)),\quad & w \notin q^{\Z},\\
\OO,\quad & w \in q^{\Z}.
\end{cases}
 \]
\item For every elliptic curve $E/\C$ there exists $q \in \C^*$ with $|q|<1$, such that 
\[ 
E_q\cong E(\C)
 \]
holds. This $q$ is given by $q=q(\tau)=\exp(2\pi i \tau)$ for the $\tau$ from theorem \ref{satz:iso_kurve_gitter}. As in theorem \ref{satz:iso_kurve_gitter} we write $\Lambda=\langle 1, \tau\rangle_\Z$. The isomorphism $\psi: E_q \rightarrow E(\C)$ satisfies $\psi=\psi_1 \circ \theta^{-1} \circ \psi_2^{-1}$, where $\psi_1$ is from theorem \ref{satz:iso_kurve_gitter} and
\[ 
\theta: \C/\Lambda\ \tilde{\rightarrow}\ \C^\ast/q^{\Z}, \quad z \mapsto w=\exp(2 \pi iz).
 \]
Hence, $\psi$ is defined via
\[ 
(x(w, q), y(w, q)) \mapsto (\wp(\alpha z, \alpha\Lambda), \wp^{\prime}(\alpha z, \alpha\Lambda)), \quad \OO \mapsto \OO,
 \]
where $\alpha$ is as in theorem \ref{satz:iso_kurve_gitter}.
\end{enumerate}
\end{satz}

As follows from the last statement, the Tate curve parametrises isomorphism classes of elliptic curves over $\C$. This is the main idea of the applications presented in section \ref{sec:abschnitt_elkies}. One can compute the objects in question, the elliptic Gau{\ss} sums from equation \eqref{eq:ell_gs}, as formal power series in $q$ by means of the Tate curve. Proposition \ref{prop:darstellung_besser} and corollary \ref{koro:gs_modulfunktion_2} show that the resulting power series admit a representation as a rational expression in terms of $j(q)$ and $m_\ell(q)$. These formulae may then be specialised to a concrete elliptic curve $E\cong E_{q(\tau)}$ over $\C$  (or over $\F_p$ after reduction) by replacing the formal variable $q$ by $q(\tau)$ as detailed in section \ref{sec:abschnitt_elkies}.\\

We now study the behaviour of $\wp(z, \tau)$ and $\wp^{\prime}(z, \tau)$ under transformations from $\Gamma$. Using lemma \ref{lem:x_y_def} this will yield results on the behaviour of $x(w, q)$ as well as $y(w, q)$. First, we derive the following
\begin{lemma}\label{lem:wp_funktion_trafo}
Let $\tau \in \mathbb{H}$ and $\gamma=\begin{pmatrix}
a & b \\ c & d
\end{pmatrix}
\in \Gamma$. Then
\begin{align}
\wp(z, \gamma\tau)&=(c\tau+d)^2\wp((c\tau+d)z, \tau),\\
\wp^{\prime}(z, \gamma\tau)&=(c\tau+d)^3\wp^{\prime}((c\tau+d)z, \tau)
\end{align}
hold.
\end{lemma}
\begin{proof}
We calculate
\begin{align*}
\wp\left(z, \frac{a\tau+b}{c\tau+d}\right)&=\frac{1}{z^2}+\sum_{n^2+m^2\neq 0}\left(\frac{1}{(z-(m+n\frac{a\tau+b}{c\tau+d}))^2} -\frac{1}{(m+n\frac{a\tau+b}{c\tau+d})^2} \right)\\
&=(c\tau+d)^2\cdot \frac{1}{((c\tau+d)z)^2}\\
&+(c\tau+d)^2\sum_{m^2+n^2\neq 0}\left( \frac{1}{((c\tau+d)z-S_{a,b,c,d}(m, n))^2} -\frac{1}{(S_{a,b,c,d}(m, n))^2} \right)\\
&=(c\tau+d)^2\wp((c\tau+d)z, \tau),
\end{align*}
where we make use of the abbreviation $S_{a,b,c,d}(m, n)=m(c\tau+d)+n(a\tau+b)$ and the last equality follows from $ad-bc=1$. 
The identity for the derivative of $\wp$ is shown analogously.
\end{proof}

\begin{koro}\label{koro:gs_modulfunktion_1}
\begin{enumerate}
Let $\gamma \in \Gamma_0(\ell)$.
\item Let $\zeta_\ell \in \mu_\ell$ be an $\ell$-th root of unity. Then we have
\[ 
x(\zeta_\ell, q(\gamma\tau))=(c\tau+d)^2x(\zeta_\ell^d, \tau)\quad \text{and}\quad y(\zeta_\ell, q(\gamma\tau))=(c\tau+d)^3y(\zeta_\ell^d, \tau).
 \]
\item The function $p_1(q)$ defined in formula \eqref{eq:p1_def_1} transforms under action of $\gamma$ according to $p_1(q(\gamma\tau))=(c\tau+d)^2p_1(q)$.
\item Let $n$ be a divisor of $\ell-1$, $\chi: \F_\ell^{*} \rightarrow \mu_n$ a Dirichlet character of order $n$ and let
\[ 
V=\begin{cases}
x,\quad n \equiv 1 \mod 2,\\
y, \quad n \equiv 0 \mod 2
\end{cases}
\text{and}\quad 
e=\begin{cases}
2n,\quad n \equiv 1 \mod 2,\\
3n, \quad n \equiv 0 \mod 2.
\end{cases}
 \]
Then the $n$-th power of the function
\begin{equation}\label{eq:G_ln_lr}
G_{\ell, n, \chi}(q)=G_{\ell, n}(q)=\sum_{\lambda \in \F_\ell^*} \chi(\lambda)V(\zeta_\ell^{\lambda}, q),
 \end{equation}
which depends on $\chi$, transforms under action of $\gamma$ via $G_{\ell, n}(q(\gamma\tau))^n=(c\tau+d)^{e}G_{\ell, n}(q)^n$.
\item The functions $p_1(q)$ and $G_{\ell, n, \chi}(q)^n$ are meromorphic at the cusps and hence modular functions of the respective weight for $\Gamma_0(\ell)$. Furthermore, $G_{\ell, n}(q)^n$ is independent from the choice of the primitive $\ell$-th root of unity $\zeta_\ell \in \mu_\ell$ and thus well-defined. 
\end{enumerate}
\begin{proof}
\begin{enumerate}
\item For $k=\frac{1}{(2\pi i )^2}$
\[ 
x(\zeta_\ell, q(\tau))=k\wp\left(\frac{v}{\ell}, \tau\right)
 \]
holds for some $v \in \Z$. This implies
\begin{align*}
x(\zeta_\ell, q(\gamma\tau))&=k\wp\left(\frac{v}{\ell}, \gamma\tau\right)=k(c\tau+d)^2\wp\left(\frac{(c\tau+d)v}{\ell},\tau\right)\\
&=(c\tau+d)^2x\left(\exp\left(\frac{2\pi i (c\tau+d)v}{\ell}\right), q\right)
=(c\tau+d)^2x\left(\exp\left(\frac{2\pi i dv}{\ell}\right), q\right)\\
&=(c\tau+d)^2x(\zeta_\ell^d, q),
\end{align*}
where the penultimate equality immediately follows from the series expansion \eqref{eq:x_def_alt} of $x(w, q)$ on taking $c \equiv 0 \mod \ell$ into account. The equation for $y$ is derived in an analogue way.
\item $\gamma \in \Gamma_0(\ell)$ implies $d\not\equiv 0 \mod \ell$, hence the action of $\gamma$ permutes the summands in equation \eqref{eq:p1_def} on simultaneous multiplication by $(c\tau+d)^2$.
\item Without loss of generality let $n$ be odd. We calculate
\begin{align*}
G_{\ell, n}(q(\gamma\tau))&=\sum_{\lambda \in \F_\ell^*} \chi(\lambda)x(\zeta_\ell^{\lambda}, q(\gamma\tau))
=(c\tau+d)^2\sum_{\lambda \in \F_\ell^*} \chi(\lambda)x(\zeta_\ell^{d\lambda}, q(\tau))\\
&=(c\tau+d)^2\chi^{-1}(d) \sum_{a \in \F_\ell^*} \chi(a)x(\zeta_\ell^{a}, q(\tau))
=(c\tau+d)^2\chi^{-1}(d)G_{\ell, n}(q).
\end{align*}
Our claim is immediate, since $\chi$ takes values in $\mu_n$.
\item We recall the set of representatives for $\Gamma/\Gamma_0(\ell)$ which consists of $\{S_k, k=0, \ldots, \ell\}$ with
\[ 
S_k=\begin{pmatrix}
0 & -1 \\ 1 & k
\end{pmatrix}\quad \text{for}\quad 0 \leq k < \ell,\quad
S_\ell=\begin{pmatrix}
1 & 0\\
0 & 1
\end{pmatrix}.
 \]  
According to remark \ref{bem:rep_system_reicht} we compute the Fourier expansion of $x(\zeta_\ell, q(S_k\tau))$ for all $S_k$. For $k<\ell$ we obtain
\begin{align*}
x(\zeta_\ell, q(S_k\tau))=(\tau+k)^2x\left(\exp\left(\frac{2\pi i (\tau+k)v}{\ell}\right), q\right)=(\tau+k)^2x(\zeta_\ell^kq^{\frac{v}{\ell}}, q).
\end{align*}
Computing $x(\zeta_\ell^kq^{\frac{v}{\ell}}, q)$ using formula \eqref{eq:x_def} we directly see the $q$-expansion contains only finitely many negative exponents. An analogue statement can be shown for $y(w, q)$. From the construction of $p_1(q)$ as well as $G_{\ell, n}(q)^n$ it now follows that these modular functions are meromorphic at the cusps.\\
The independence of $G_{\ell, n}(q)^n$ from the choice of $\zeta_\ell$ is an immediate consequence of what we have already shown, since for $d \not\equiv 0 \mod \ell$ this implies
\[ 
\sum_{\lambda \in \F_\ell^*} \chi(\lambda)x(\zeta_\ell^{d\lambda}, q)=
\chi^{-1}(d)\sum_{\lambda \in \F_\ell^*} \chi(\lambda)x(\zeta_\ell^{\lambda}, q).
 \]
\end{enumerate}
\end{proof}
\end{koro}

We finally construct the \textit{universal elliptic Gau{\ss} sums}.

\begin{koro}\label{koro:gs_modulfunktion_2}
Let $\ell$, $n$, $\chi$ be as in corollary \ref{koro:gs_modulfunktion_1}. Furthermore, let
\[ 
r=
\begin{cases}
\min\left\{r: \frac{n+r}{6} \in \N\right\},& n \equiv 1 \mod 2,\\ 
3,& n=2, \\
0,& \text{else},
\end{cases}
\text{and}\quad 
e_{\Delta}=\begin{cases}
\frac{n+r}{6}, & n\equiv 1 \mod 2,\\
1, & n=2,\\
\frac{n}{4}, & \text{else.}
\end{cases}
 \]
Then
\begin{equation}\label{eq:tau_ln_lr}
\tau_{\ell, n}(q):=\frac{G_{\ell, n}(q)^n p_1(q)^r}{\Delta(q)^{e_{\Delta}}}
 \end{equation}
is a modular function of weight $0$ for $\Gamma_0(\ell)$ which is holomorphic on $\mathbb{H}$ and whose coefficients lie in $\Q[\zeta_n]$.
\end{koro}
\begin{proof}
Without loss of generality let $n$ be odd. Corollary \ref{koro:gs_modulfunktion_1} shows that
$G_{\ell, n}(q)^n$ and $p_1(q)$ are modular functions for $\Gamma_0(\ell)$ of weight $2n$ and $2$, respectively. Furthermore, it is known that the discriminant is a modular function of weight $12$ for $\Gamma$, hence also for subgroups. Since in addition $\Delta(q)\neq 0$ holds, the expression in question is likewise a modular function for $\Gamma_0(\ell)$ whose weight can obviously be computed to be
\[ 
2n+2r-12\frac{n+r}{6}=0.
 \]
Using formula \eqref{eq:x_def_alt} the poles of $x(\zeta_\ell, q(\tau))$ can easily be determined. Obviously, there is a pole if and only if $q^n=1$ or $q^n=\zeta_\ell$ holds for some $n$. This implies $|q|=|\exp(2\pi i \tau)|=1$, which contradicts $\tau \in \mathbb{H}$. Thus, the functions in the enumerator are holomorphic by construction. An analogue statement can be shown when using $y(\zeta_\ell, q(\tau))$. Furthermore, $\Delta(\tau)\neq 0$ holds on $\mathbb{H}$, which implies the assertion.\\
In order to prove the statement concerning the coefficients we first remark that the coefficients of $p_1(q)$ as well as those of $\Delta(q)$ lie in $\Z$. As the leading coefficient of $\Delta$ is $1$, this also holds for the coefficients of $\Delta(q)^{-1}$. In addition it is evident from the definition of the Gau{\ss} sum in corollary \ref{koro:gs_modulfunktion_1} that its coefficients lie in $\Q[\zeta_n, \zeta_\ell]$. Now let $c$ be a generator of $\F_\ell^*$ and $\sigma: \zeta_\ell \mapsto \zeta_\ell^c$, such that $\gal(\Q[\zeta_n, \zeta_\ell]/\Q[\zeta_n])$ is generated by $\sigma$. We calculate
\begin{align*} 
\sigma(G_{\ell, n}(q))&=\sum_{\lambda \in \F_\ell^*} \chi(\lambda) \sigma(V(\zeta_\ell^{\lambda}, q))
=\sum_{\lambda \in \F_\ell^*}\chi(\lambda)V(\zeta_\ell^{c\lambda}, q)\\
&=\chi^{-1}(c)\sum_{\lambda \in \F_\ell^*}\chi(c\lambda)V(\zeta_\ell^{c\lambda}, q)
=\chi^{-1}(c)G_{\ell, n}(q),
 \end{align*}
where the second equality follows from formula \eqref{eq:x_def} and \eqref{eq:y_def}, respectively, since $\sigma$ is a homomorphism. We deduce immediately that $G_{\ell, n}(q)^n$ is invariant under action of $\sigma$, which in turn implies the coefficients of the expression lie in $\Q[\zeta_n]$. 
\end{proof}

\section{Application to point-counting}\label{sec:abschnitt_elkies}

Since $\tau_{\ell, n}(q)$ fulfils the requirements of proposition \ref{prop:darstellung_besser}, the universal elliptic Gau{\ss} sum may be represented in terms of $j(q)$ and $m_\ell(q)$. We now present an approach how this can be used to determine the order of the group of points on a concrete elliptic curve $E: Y^2=X^3+aX+b$ over a field $\F_p$.\\

We assume that the $j$-invariant of the curve $E$ in question is different from $0$ and $1728$ and we consider Elkies primes $\ell$, for which the characteristic equation \eqref{eq:char_gl} of the Frobenius homomorphism $\phi_p$ factors into linear factors over $\F_\ell$. We denote a maximal prime power divisor of  $\ell-1$ by $n$ and $\chi: (\Z/\ell\Z)^* \rightarrow \langle \zeta_n \rangle$ denotes a character of order $n$. Further let $\lambda$ be an eigenvalue of $\phi_p$ modulo $\ell$, which yields $t\equiv \lambda+ \frac{p}{\lambda} \mod \ell$. Using the elliptic Gau{\ss} sum \eqref{eq:ell_gs} from \cite{Mi1} we obtain the identity 
\begin{equation}\label{eq:gleichung_konkret_2}
\frac{G_{\ell, n, \chi}(E)^m}{G_{\ell, n, \chi^m}(E)}(G_{\ell, n, \chi}(E)^n)^q=\chi^{-m}(\lambda),
\end{equation}
as presented in equation \eqref{eq:lambda_aus_ell_gs}, where $p=nq+m$ holds.\\

Assuming we have precomputed the universal elliptic Gau{\ss} sums we dispose of the value
\begin{equation}
\tau_{\ell, n}(q)=\frac{G_{\ell, n, \chi}(q)^n p_1^r(q)}{\Delta(q)^{e_{\Delta}}}
\end{equation}
as a rational expression $R$ in terms of the Laurent series $m_\ell(q)$ and $j(q)$ according to proposition \ref{prop:darstellung_besser}, where $r$ and $e_{\Delta}$ are defined in corollary \ref{koro:gs_modulfunktion_2}.\\
By means of the Deuring lifting theorem from \cite[p.~184]{Lang} we lift the curve $E/\F_p$ in question to a curve $E_0$ over a number field $K$. This means there is a prime ideal $\mathfrak{P} \subset \OO_K$ with residue field $\F_p$ such that the reduction of $E_0$ modulo $\mathfrak{P}$ is a non-singular elliptic curve which is isomorphic to $E$. Now using the specialisation $q \mapsto q(\tau)=\exp(2\pi i \tau)$, where $\tau$ denotes the value $\tau$ associated to $E_0$ from theorem \ref{satz:iso_kurve_gitter}, we transfer the representation for $\tau_{\ell, n}(q)$ we computed by means of the Tate curve to the special curve $E_{q(\tau)}$. Subsequently, we make use of the isomorphism $\psi: E_{q(\tau)} \tilde{\rightarrow}\ E_0$ defined in theorem \ref{satz:tate_kurve_eig}. Taking into account the transition from multiplicative to additive structure (cf. theorem \ref{satz:tate_kurve_eig}), we find that for $a \in \Z$
\[ 
\psi(x(\zeta_\ell^a, q(\tau)))=(aP)_x,\quad \psi(y(\zeta_\ell^a, q(\tau)))=(aP)_y 
 \]
holds for an $\ell$-torsion point $P$ on $E_0$. Furthermore $\Delta(q(\tau))=\Delta\left(E_{q(\tau)}\right)$ corresponds to the value $\Delta(E_0)$ and $j(q(\tau))=j\left(E_{q(\tau)}\right)$ to the quantity $j(E_0)$.
This implies the equality
$\psi(G_{\ell, n, \chi}(q(\tau)))=G_{\ell, n, \chi}(E_0)$ 
holds and we can use the rational expression for $\tau_{\ell, n}$, which we assume to have been precomputed in a general setting, to determine the elliptic Gau{\ss} sum  $G_{\ell, n, \chi}(E_0)$ and, as will now be shown, also $G_{\ell, n, \chi}(E)$.\\

Since in the cases we are interested in $\ell=O(\log p)$ holds, formulae \eqref{eq:x_def} and \eqref{eq:y_def} as well as the equality of ideals $(\ell)=(1-\zeta_\ell)^{\ell-1}$ in $\Q[\zeta_\ell]$ imply the denominators of the coefficients of the Laurent series are invertible modulo $p$ and thus all calculations can be reduced modulo the prime ideal $\mathfrak{P}\mid(p)$. Hence, we finally obtain the formula
\begin{equation}\label{eq:gausssumme_konkret}
\frac{G_{\ell, n, \chi}(E)^n p_1(E)^r}{\Delta(E)^{e_{\Delta}}}=R(m_\ell(E), j(E))
\end{equation}
on $E/\F_p$. The values are now those associated to the curve $E$ in question. These quantities on $E$ can be determined using the well-known formulae for $\Delta(E)$ and $j(E)$ as well as  the formulae to compute $p_1(E)$ as derived e.~g. in \cite[pp.~269--271]{Morain}.\\

Here the value of $p_1(E)$ depends on $m_\ell(E)$, which is found as a root of the polynomial $M_\ell(X, j(E))$, obtained from the precomputed polynomial $M_\ell(X, Y)$. If $\ell$ is an Elkies prime, $m_\ell(E)$ lies in $\F_p$. The possible values for $m_\ell$ correspond to the eigenvalues $\lambda$ and $\mu$ of $\phi_p$ (cf. section \ref{sec:ell_kurven}). Since the polynomial $M_\ell(X, j(E))$ does not have a double root according to \cite[p.~109]{Mueller}, we obtain two possibilities for $m_\ell(E)$.\\

Now equation \eqref{eq:gausssumme_konkret} can be transformed to
\begin{equation}\label{eq:gaussssumme_konkret_2}
G_{\ell, n, \chi}(E)^n =\frac{R(m_\ell(E), j(E)) \Delta(E)^{e_{\Delta}}}{p_1(E)^r}.
 \end{equation}

\begin{bem}
\begin{enumerate}
\item If the $n$-th cyclotomic polynomial over $\F_p$ is irreducible, the adjustment by multiplying by $p_1(E)$, $\Delta(E)$ on the curve can be omitted, as these values lie in $\F_p$ and the roots of unity are $\F_p$-linearly independent. This allows in particular to save the cost for computing  $p_1(E)$, which is non-negligible. However, this condition is evidently not met for $n=2$, which means we cannot avoid the computation of $p_1(E)$ (or rather $p_1(E)^3$). Some time might be saved, though, provided $\ell-1 \equiv 0 \mod 4$ holds.

\item According to proposition \ref{prop:darstellung_besser} the quantity
\[ 
\frac{\partial M_\ell}{\partial Y}(m_\ell(E), j(E))
 \]
is used as the denominator of the rational expression $R$. We remark that this value may be zero in isolated cases, as was already mentioned in \cite[p. 110]{Mueller}. In such a case one can usually compute the expression using the second possible value for $m_\ell(E)$. If this yields zero as well, the prime $\ell$ cannot be used within our method.\\

We have to guarantee the value $p_1$ occurring in the denominator in equation \eqref{eq:gaussssumme_konkret_2} does not vanish modulo $p$. Since we consider curves whose $j$-invariants do not equal $0$ nor $1728$,
we obtain $a\cdot b \neq 0$.
Furthermore, one can show that the constant term of $M_\ell(X, j(E))$ exhibits the value $\ell^s$ and is thus different from zero modulo $p$.
Hence, $m_\ell(E) \neq 0$ follows and using the formulae from \cite[pp.~270]{Morain} we conclude
\[ 
p_1=0 \Leftrightarrow \frac{\partial M_\ell}{\partial Y}(m_\ell(E), j(E))=0.
 \]
So formula \eqref{eq:gaussssumme_konkret_2} can be applied whenever the rational expression for a value of $m_\ell(E)$ can be determined.
\end{enumerate}
\label{bem:elkies_m_nenner_nicht_null}
\end{bem}

Now it suffices to compute the $q$-th power $(G_{\ell, n}(E)^n)^q$ in order to determine the discrete logarithm of $\lambda$ in $(\Z/\ell\Z)^*$ modulo $n$ for primes $p \equiv 1 \mod n$ using equation \eqref{eq:gleichung_konkret_2}. Generalizing the universal elliptic Gau{\ss} sums in order to precompute the elliptic Jacobi sums $\frac{G_{\ell, n, \chi}(E)^m}{G_{\ell, n, \chi^m}(E)}$ from equation \eqref{eq:gleichung_konkret_2} we can determine $\lambda$ for arbitrary large primes $p$, as will be shown in a separate paper. The modular information may then be composed to derive the values of $t$ and $\#E(\F_p)$ as in Schoof's algorithm.
\footnotesize{
\bibliography{lit}}

\end{document}